\documentclass[10pt]{article}
\usepackage[usenames,dvipsnames]{color}
\makeatletter
\providecommand{\@LN}[2]{}
\makeatother
\usepackage{mathtools}
\usepackage{tikz}
 \usepackage{esvect}
 \usepackage{makecell}
  \usepackage{color}
\usepackage{hyperref}

\usepackage{amsmath}
\usepackage{amssymb}
\usepackage{amsthm}
\usepackage{amscd}
\usepackage[mathscr]{eucal}
\usepackage{url}
\usepackage{ascmac}
\usepackage{array}
\usepackage[all]{xy}
\usepackage{multirow,eepic}
\usepackage{enumitem}
\usepackage{ulem}
\usepackage{extarrows}
\newcommand\phantomarrow[2]{%
  \setbox0=\hbox{$\displaystyle #1\to$}%
  \hbox to \wd0{%
    $#2\mapstochar
     \cleaders\hbox{$\mkern-1mu\relbar\mkern-3mu$}\hfill
     \mkern-7mu\rightarrow$}%
\,}

\def\senbun#1(#2)#3({\@senbun(#2)(}
\def\@senbun(#1,#2)(#3,#4){%
   \@tempdima#1\p@ \advance\@tempdima#3\p@
   \divide\@tempdima\tw@
   \@tempdimb#2\p@ \advance\@tempdimb#4\p@
   \divide\@tempdimb\tw@
   \edef\@senbun@temp{\noexpand\qbezier(#1,#2)%
      (\strip@pt\@tempdima,\strip@pt\@tempdimb)(#3,#4)}%
   \@senbun@temp}
\makeatother

\DeclareMathOperator{\tr}{tr}

\DeclareMathOperator{\End}{End}

\DeclareMathOperator{\id}{id}

\DeclareMathOperator{\ch}{ch}

\newcommand{\MG}{SL_2(\Z)}
\newcommand{\Vir}{\mathsf{Vir}}
\newcommand{\f}{\mathsf{f}}

\newcommand{\blue}[1]{{\color{blue}#1}}
\newcommand{\C}{\mathbf{C}}
\newcommand{\Z}{\mathbf{Z}}

\newcommand{\R}{\mathbb{R}}
\newcommand{\N}{\mathbf{N}}
\newcommand{\Q}{\mathbf{Q}}

\renewcommand{\H}{\mathfrak{H}}
\newcommand{\VF}{\mathsf{F}}

\newcommand{\unit}{\mathbf{1}}

\renewcommand{\=}{\,=\,}

\newcommand{\NO}{\,{\raise0.25em\hbox{$\mathop{\hphantom{\cdot}}%
\limits^{_{\circ}}_{^{\circ}}$}}\,}

\theoremstyle{plain}
   \newtheorem{theorem}{Theorem}
   
   \newtheorem*{corollary}{Corollary}
   
   \newtheorem{lemma}[theorem]{Lemma}
   \newtheorem{proposition}[theorem]{Proposition}
   
   \theoremstyle{definition}
   
   \newtheorem*{definition}{Definition}
     
     \newtheorem*{notations}{Notations}


\newcommand{\G}{\Gamma}
\newcommand{\g}{\gamma}

\newcommand{\sd}{\mathfrak{d}}
\newcommand{\mn}{\medskip\noindent}

\newcommand{\Hol}{\operatorname{\textsf{Hol}}}
\newcommand{\bt}{\begin{theorem}}
\newcommand{\et}{\end{theorem}}

\makeatletter
\newcommand{\textarrow}[2][1]
  { \settowidth{\@tempdima}{#2}
    \stackrel{#2}
             {\makebox[#1\@tempdima][l]{\rightarrowfill}}
  }
\makeatother

\def\be{\begin{equation}}   \def\ee{\end{equation}}     \def\bes{\begin{equation*}}    \def\ees{\end{equation*}}
\def\ba{\be\begin{aligned}} \def\ea{\end{aligned}\ee}   \def\bas{\bes\begin{aligned}}  \def\eas{\end{aligned}\ees}

\renewcommand{\=}{\,=\,}

\newcommand{\ec}{c_{\operatorname{eff}}}
\newcommand{\Sym}{\operatorname{Sym}}

\newcommand{\sch}{\mathfrak{ch}}

\renewcommand{\mod}[1]{\operatorname{mod}{#1}}
\renewcommand{\pmod}[1]{\,(\operatorname{mod}{#1})}
\newcommand{\e}[2]{e^{\frac{#1\pi i}{#2}}}
\renewcommand{\t}{\tau}
\makeatletter
\newsavebox\myboxA
\newsavebox\myboxB
\newlength\mylenA

\def\thin{{\hskip 1pt}} 

\newcommand*\xoverline[2][0.75]{%
    \sbox{\myboxA}{$\m@th#2$}%
    \setbox\myboxB\null
    \ht\myboxB=\ht\myboxA%
    \dp\myboxB=\dp\myboxA%
    \wd\myboxB=#1\wd\myboxA
    \sbox\myboxB{$\m@th\overline{\copy\myboxB}$}
    \setlength\mylenA{\the\wd\myboxA}
    \addtolength\mylenA{-\the\wd\myboxB}%
    \ifdim\wd\myboxB<\wd\myboxA%
       \rlap{\hskip 0.5\mylenA\usebox\myboxB}{\usebox\myboxA}%
    \else
        \hskip -0.5\mylenA\rlap{\usebox\myboxA}{\hskip 0.5\mylenA\usebox\myboxB}%
    \fi}
\makeatother
\renewcommand{\R}{\mathbf{R}}

\begin{document}

\begin{Large}
\begin{center}
Minimal models, modular linear differential equations\\and\\modular forms of fractional weights
\end{center}
\end{Large}

\mn
\begin{center}
Kiyokazu~Nagatomo\footnote{The first author is supported in part  by International Center of Theoretical Physics, Italy, 
and Max-Planck institute for Mathematics, Germany.} and Yuichi~Sakai\footnote{The second author is supported in part 
by JSPS KAKENHI Grant numbers JP21K03183, JP18K03215 and JP16H06336.} 

\setcounter{footnote}{0}

\medskip
${}^1$ Department of Pure and Applied Mathematics\\ 
Graduate School of Information Science and Technology\\
Osaka University, Suita, Osaka 565-0871, Japan

\mn
${}^2$
Fundamental Educational Center\\
Kurume Institute of Technology\\
2228-66, Kamitsu, Kurume, Fukuoka 830-0052, Japan
\\e-mail: dynamixaxs@gmail.com
\end{center}

\begin{abstract} 
We show that modular forms of fractional weights on principal congruence subgroups of odd levels,  which are found by T.~Ibukiyama,
naturally appear as characters being multiplied~$\eta^{\ec}$ of the so-called minimal models of type~$(2, p)$, where~$\ec$ is 
the effective central charge
of a minimal model. Using this fact and modular invariance property of the space of characters, we give a different proof of the result showned by
Ibukiyama that is the explicit formula representing~$\MG$ on the space of the ibukiyama modular forms. 
We also find several pairs of spaces of the Ibukiyama modular forms on~$\G(m)$ and~$\G(n)$ with~$m|n$ having the property 
that the former are included in the latter.
Finally, we construct vector-valued modular forms of weight~$k\in\frac{1}{5}\Z_{>0}$ starting from the Ibukiyama modular forms 
of weight~$k\in\frac{1}{5}\Z_{>0}$ with some multiplier
system by a symmetric tensor product of this representation and show that the components functions coincides with the solution space of some
monic modular linear differential equation of weight~$k$ and the order~$1+5k$.
\end{abstract}

\section*{Introduction}
In this paper we study modular forms of fractional weights found by T.~Ibukiyama~\cite{I} from the conformal field theoretic point of view. 
The principal goal of this paper is to relate these modular forms to the characters of the so-called \textit{minimal model} of type~$(2, p)$ 
with an odd integer~$p>3$, which is a well-known example of 2D~conformal field theories. It was proved in~\cite{Milas1} that the space of characters 
of any minimal model of type~$(P,Q)$ is equal to the solution space of some \textit{monic linear differential equation} of order~$(P-1)(Q-1)/2$. 
We shows, in this paper, that the Ibukiyama modular forms on $\G(p)$ satisfy some monic linear \textit{modular} differential equation (MLDE) 
of the same order~$(p-1)/2$ 
but a different weight.  The space of the Ibukiyama modular forms, where~$v_p$ is 
a certain multiplier system, is equal to the space~$M_{\frac{p-3}{2p}}(\G(p),v_p)$ of modular forms of weights~$\frac{p-3}{2p}$ on~$\G(p)$
for any odd integer~$p<15$. (See Lemma~1 and Lemma~2 in section~3 in~\cite{I2}.)  However, for $p=15$ and~23, it was shown in section~4 
of~\cite{I2} that $M_{\frac{p-3}{2p}}(\G(p),v_p)$ do not coincide with the space of Ibukiyama modular forms. The dimension is  known \cite{I} Lemma 1.7.

It is known that the character~$\ch_M$ of an arbitrary module~$M$ over a rational simple self-dual vertex operator algebra (VOA)~$V$ 
of \textit{CFT} and 
of \textit{finite type} ($C_2$-cofinite) is a modular function on the congruence subgroup~$\G(n)$ of~$\MG$, where~$n$ is the minimal positive integer\footnote{See equation~\eqref{eq.n} 
in~Thereom~\ref{t.n} for the general  formula of~$n=n_p$ of the minimal model of type~$(2, p)$.} 
such that~$n(h-c/24)$ is integral for \textit{any conformal weight~$h$} of all (inequivalent) simple~$V$-modules, 
where~$c$ is the central charge of~$V$. (See Theorem~I and Theorem~3.10 in~\cite{DLN}.) The central charges and conformal weights 
of a simple VOA of finite type are rational numbers, and number of their modules is finite.)  Let~$\ec$ be the \textit{effective central charge} 
defined by~$\ec=c-24\thin h_{\min}$, where~$h_{\min}$ is the minimal element of the (finite) set of conformal weights.  
Then~$\eta^{\ec}\ch_M$ is a modular function, where $\ch_M$ is a character of the minimal model of type $(2, p)$
with the multiplier system that of~$\eta^{\ec}$ on~$\G(n)$.  For instance, for the minimal model type~$(2,5)$ we have~$n_5=60$ and~$\ec=2/5$. 
Therefore~$\eta^{2/5}\ch_M$ is a modular function on~$\G(60)$ with the multiplier of~$\eta^{2/5}$. 
We will show that the Ibukiyama modular form of weight~$\frac{p-3}{2p}$ on the congruence group~$\G(p)$ with some multiplier system~$v_p$ 
equals~$\eta^{\ec}\ch_M$. It implies that~$\ch_M$ is a modular form with some non-trivial multiplier system on~$\G(p)$ since the multiplier system 
of the Ibukiyama modular forms on $\G(p)$ is  not equal to that of~$\eta^{\ec}$ on~$\G(p)$. 

There are at least two interesting problems of modular forms of fractional weight. (a) Any $f\in M_{*}(\G(p),\nu)$ is \textit{standard},  where~$\nu$ is a multiplier, i.e.,  whether there exists a positive integer~$m$ such that~$f^m\in M_{*}(\G(p),\unit)$. An arbitrary Ibukiyama modular form is proved to be standard 
(Theorem~1.1~(3), Section 1.1) in~\cite{I}.
However, it is not known whether modular functions~$\eta^{\ec}\ch_M$ for \textit{any minimal models} always satisfy this important  property. 
Of course they are standard on~$\G(24)\cap\G(n)$.
(b)~let~$M_k(\G(m),\mu)$ and~$M_k(\G(n),\nu)$ be spaces of modular forms of weight~$k$ on~$\G(m)$ and~$\G(n)$ with
 multiplier systems~$\mu$ and~$\nu$ satisfying~$\nu|_{\G(m)}=\mu$ for $m|n$. The question is that if~$m|n$ we have $M_k(\G(m),\mu)\subset M_k(\G(n),\nu)$. 
Note that any spaces of modular forms of (half-)\thin\textit{integral weight} of congruence groups of~$SL_2(\R)$ have the property~(b).  
In this paper we show that the pair of the spaces~$M_*(\G(5),v_5):=\bigoplus_{m=0}^\infty M_{\frac{m}{5}}(\G(5), v_5^{m})$  
and~$M_*(\G(15), v_{15}):=\bigoplus_{m=0}^\infty M_{\frac{2m}{5}}(\G(15), v_{15}^{m})$  
and~$M_*(\G(7), v_7):=\bigoplus_{m=0}^\infty M_{2m/7}(\G(7), v_{7}^{m})$ 
and $M_*(\G(21), v_{21}): =\bigoplus_{m=0}^\infty M_{3m/7}(\G(21),v_{21}^m)$ have the property~(a).

Let~$f$ be an arbitrary (holomorphic or meromorphic) modular form of weight~$k$ on a Fuchsian group $\G\subset PSL_2(\R)$ such that~$\H\backslash\G$ is genus~0. 
Then~$f$ satisfies a certain ordinary equation of order~$k+1$ with a local coordinate defined by a modular invariant function (Hauptmoduln)
(e.g.~p.~61,  Propposition~21~\cite{123}) Let~$k\in\frac{1}{5}\Z_{>0}$. We construct  
 vector-valued modular forms by using the symmetric representation of an action~$\MG$ on~$M_k(\G(5), v_5^{5k})$, and prove that the space spanned by 
 components function of this VVMF is equal to the solution space of some MLDE is of weight~$k$ and order~$1+5k$. Namely, vector-valued modular forms, which are constructed by fractional modular forms, provide examples of 
differential equations whose orders are higher than the weight\,$+1$.

In~\S\ref{s.imf} we review the notion of multipliers systems, and introduce the multiplier systems which are used to define the Ibukiyama modular forms. 
The definition 
of the Ibukiyama modular forms is also given, and the invariance of the space of the Ibukiyama modular forms 
under some action~$\pi_p$ of~$SL_2(\Z)$ is mentioned in the same section. We  briefly explain the notion of vertex operator algebras in~\S\ref{S.VOA}.
The minimal models are introduced in~\S\ref{s.minimal_models}. Here we collect several facts of characters of minimal models which are used in the following sections. 
We further prove the formula that gives an explicit formula of~$n_p\,(\,=n)$ of the minimal model of type~$(2,p)$.
 We clarify in~\S\ref{s.minimal_Ibukiyama} the relation between the Ibukiyama modular forms  and characters of the minimal models of type~$(2,p)$ 
in Theorem~\ref{t.minimal_Ibukiyama}. We then give a different proof of modular invariance of the Ibikuyama modular forms under the action 
of~$\pi_p$ by using the modular invariance property of the space of characters of minimal models. In~\S\ref{MLDE} we derive MLDEs whose solution spaces are equal to the space of characters of the minimal models of type~$(2,p)$ and then the space of the Ibukiyama modular forms, respectively. However, it is shown that they are immediate consequence of differential equations of Milas, 
and we \textit{cannot} obtain any new way to find: using the fractional  modular forms supersingular $j$-invariant as they did.  Here we give several examples which satisfy the property~(a).
In the final section~\ref{s.monodromy} we study the case~$p=5$.  We compute  an action of $\MG$ on a basis of $M_{\frac{1}{5}}(\G(5),v_5)$  
and show that vector-valued modular forms the space of the symmetric tensor representation of degree~$5k$ associated with this representation above is equal to the solution space of some MLDE of the order~$1+5k$.

\section{Ibukiyama's modular forms}\label{s.imf}
In this section we review the Ibukiyama modular forms which are main interest of this paper.  Let denote~$e(z)=e^{2\pi iz}$ 
with any complex number~$z$. For any element~$\vv{m}=(m_1, m_2)\in\Q^2$ we define the theta function in the complex upper half-plain~$\H$ 
of \textit{characteristic}~$\vv{m}$ by
\begin{equation*}
\Theta_{\vv{m}}(\tau, z)\=\sum_{n\in\Z}e\Bigl(\frac{1}{2}\tau(n+m_1)^2+(n+m_1)(z+m_2)\Bigr)
\end{equation*}
for any~$\tau\in\H$ and any~$z\in\C$. We denote the value of~$\Theta_{\vv{m}}(\tau, z)$ at~$z=0$ by~$\Theta_{\vv{m}}(\tau)$, which is called 
the \textit{theta constant}. Let~$\ell>0$ be an odd prime number and $\varepsilon_1,\,\varepsilon_2$ be odd integers. Let $m_1=\varepsilon_1/\ell$ and~$m_2=\varepsilon_2/\ell$. Then and $\Theta_{\vv{m}}(\tau, z)$  is a modular form of weight~$1/2$ on~$\Gamma(\ell)$~(cf.~\cite{FK}).

Before giving  definition of the multilinear systems  we need 
the notion of ``automorphy factor."  Let~$\G\subset PSL_2(\Z)$ be a Fuchsian group.
An \textit{automorphy factor}~$j(\g,\tau)$ is a complex-valued function in~$\G\times\H$, which is holomorphic in the complex upper half-plane
with~$j(\g_1\g_2, \tau)=j(\g_1, \g_2\tau)j(\g_2, \tau)$ for any~$\g_1,\,\g_2\in\G$.

{Let~$v:\G\to\C^\times$ $(\g\mapsto v(\g))$ be a map and let~$r$ be a rational number. For 
any~$\g=\left(\begin{smallmatrix}a&b\\c&d\end{smallmatrix}\right)\in\G$, let denote~$j(\g,\tau)=v(\g)(c\tau+d)^r$, where we take 
the principal value for the branch of the rational power. We say that~$v(\g)$ is a \textit{multiplier system}
if~$j(\g_1\g_2,\tau) = j(\g_1, \g_2\tau)j(\g_2,\tau)$ for any~$\g_1, \g_2\in\G$, and~$j(\g, \t)$ is an automorphy factor of weight~$r$. 
We say that a holomorphic function~$f(\t)$ on~$\H$ is a holomorphic modular form of weight $r$ with a multiplier system~$v(\g)$
if~$f(\g\t) = f(\t)j(\g,\tau)$ for each~$\g\in\G$ and~$f$ is holomorphic at each cusp of~$\G$. }

We denote by~$v_0:\MG\to\C^\times$ the multiplier system 
of~$\eta(\tau)^{3/p}$ for any odd integer~$p>3$ and for any~$\gamma=\begin{psmallmatrix}a&b\\ c&d\end{psmallmatrix}\in\MG$, i.e.~$\eta(\g\tau)^{3/p}=v_0(\g)\eta(\tau)^{3/p}(c\tau+d)^{3/2p}$, and set~$v_p=v_0^{\,p^2-1}$. Then~$v_p$ is a multiplier system.

The next theorem is proved in~\cite{I}.
\begin{theorem}[Theorem~1.2 in~\cite{I}]
Let~$p>3$ be an odd integer. Then~$v_p=v_0^{\,p^2-1}$ is a multiplier system on~$\G(p)$ 
and~$j_p(\g,\tau)=v_p(\gamma)(c\tau+d)^{\frac{p-3}{2p}}$ is an automorphy factor of degree~$\frac{p-3}{2p}$.
\end{theorem}

\begin{definition}
Let~$f$ be a modular form of non-integral weight~$k$ on~$\Gamma(n)$. Then~$f$ is called \textit{standard} if there is a positive integer~$m$   
such that~$f^m$ is a modular form of integral weight with a trivial multiplier on some group including  $\G_n$.
\end{definition}

We can now give the definition of the Ibukiyama modular forms.

\begin{definition}
Let~$p>2$ be an odd integer. For any odd integer~$1\leq r\leq p-2$, the \textit{Ibukiyama modular forms}~$f_r(\tau)$, which are holomorphic function
in the complex upper-half plane~$\H$,  are defined by 
\begin{equation*}
f_r(\tau)\=\frac{\Theta_{\vv{m}}(p\thin \tau)}{\eta(\tau)^{3/p}}\,e\!\left(\frac{(p-1)(r-1)}{4p}\right)\,,\quad\vv{m}_r
\=\left(\frac{r}{2p}\,,\,\frac{1}{2}\right)\,.
\end{equation*}
\end{definition}

\begin{theorem}[\cite{I}]
Let $p>3$ be an odd integer.\\
\textup{(a)} 
For any odd integer~$r$ with~$1\leq r\leq p-2$, the Ibukiyama modular form~$f_r$ is a holomorphic modular form of weight~$\frac{p-3}{2p}$ on~$\G(p)$ 
with the multiplier system~$v_p$.\\
\textup{(b)} 
The set of the functions~$\{f_r\}_{r=1}^{p-2}$ is linearly independent.\\
\textup{(c)} 
We have~$v_p(\g)^p=1$ for any~$\g\in\G(p)$. In particular,~$(f_r)^p$ is a holomorphic modular form of integral weight~$\frac{p-3}{2}$ 
on~$\G(p)$ for any~$1\leq r\leq p-2$, and then any modular form~$f_r$ is standard. 
\end{theorem}
\noindent
In~\S\ref{s.minimal_Ibukiyama} we relate the Ibukiyama modular forms to characters of the minimal model of type~$(2, p)$ with an arbitrary odd integer~$p>3$. 
We denote the space of modular forms with multiplier systems~$v_p$ on~$\G(p)$ 
by~$M_*(\G(p),v_p)=\bigoplus_{m=0}^{\infty}M_{\frac{m(p-3)}{2p}}(\G(p),v_p^m)$.
However, at the moment, we cannot give a different proof, by using the relation to the minimal model,  that  $f_r$ is  standard on $\G(p)$.

\begin{definition} 
Let~$p>3$ be an odd integer. We define a left action~$\pi_p$ of~$\MG$ on~$M_{\frac{p-3}{2p}}(\G(p),v_p)$ 
by~$\pi_p(\g)f(\tau)=j_p(\g,\tau)^{-1}f(\g(\tau))$ for any~$f\in M_{\frac{p-3}{2p}}(\G(p),v_p)$ and any~$\g\in\MG$. 
\end{definition}
 
\begin{theorem}[Theorem~1.3 of~\cite{I}]\label{t.im}
Let $p>3$ be an odd integer.\\
\textup{(a)}
The function~$j_p$ is a factor of automorphy of~$\MG$, and is a prolongation of the factor of the automorphy of~$\G(p)$ of weight $(p-3)/2$
satisfying $j_p(-I_2, \tau)=(-1)^{(p+1)/2}$.\\
\textup{(b)}
The space of the Ibukiyama modular forms is spanned by~$f_r$ with all odd integers~$1\leq r\leq p-2$ is invariant 
under the action~$\pi_p$ of~$\MG$ as
\begin{equation}\label{eq.pi_s}
\begin{aligned}
\pi_p(S)f_r&\=\sum_{\substack{1\leq s\leq p-2\,,\\ s\equiv1\mod{2}}}\frac{1}{\sqrt{p}}\Bigl[e\Bigl(\frac{r-s}{4}+\frac{rs}{4p}+\frac{3p-1}{8}\Bigr)
+(-1)^{\frac{p+1}{2}}\,e\Bigl(-\frac{r-s}{4}-\frac{rs}{4p}-\frac{3p-1}{8}\Bigr)\Bigr]f_s\,,\\
\pi_p(T)f_r&\=e\Bigl(\frac{r^2-p^2}{8p}\Bigr)f_r\,,\quad\pi_p(-I_2)f_r\=(-1)^{(p+1)/2}f_r\,,
\end{aligned}
\end{equation}
where~$S=\left(\begin{smallmatrix}0&-1\\1&0\end{smallmatrix}\right)$, $T=\left(\begin{smallmatrix}1&1\\0&1\end{smallmatrix}\right)$
and~$I_2=\left(\begin{smallmatrix}1&0\\0&1\end{smallmatrix}\right)$.
\end{theorem}
\noindent
We will show later that equation~\eqref{eq.pi_s} is nothing but the modular invariance of the space of characters of the minimal model of type~$(2,p)$ 
in~\S\ref{s.minimal_Ibukiyama}

\section{Vertex operator algebras}\label{S.VOA}
The goal of this paper is to show that any Ibukiyama modular form of a fractional weight is equal to some character 
of a minimal model model of type~$(2,p)$ multiplied by some power of the eta function. (See Theorem~\ref{t.minimal_Ibukiyama} 
of~\S\ref{s.minimal_Ibukiyama}.) We will therefore explain the minimal models and their characters. To this end, 
we first collect definitions and basic properties related to vertex operator algebras. 

A \textit{vertex operator algebra} (VOA) is a quadruple~$(V,\,Y,\,\unit,\,\omega)$ with the following properties. The vector space~$V$ over the complex number field~$\C$
is graded as~$V=\bigoplus_{n=0}^\infty V_n$ with~$\dim  V_n<\infty$ for all~$n$. The element~$\unit\in V_0$ and~$\omega\in V_2$ are called the \textit{vacuum vector} 
and the \textit{Virasoro vector}, respectively, and a linear map~$Y:V\to\End_\C(V)[[z,z^{-1}]]$~($a\mapsto Y(a,z)=\sum_{n\in \Z}a_n z^{-n-1}$) satisfies a number of axioms. 
(For further information about axioms, see~\cite{LL} and~\cite{MN}.) We define~$L_n\in\End(V)$ for any integer~$n$ by~$Y(\omega, z)=\sum_{n\in\Z}L_nz^{-n-2}$. 
Then the set~$\{L_n,\,\id_V\mid n\in\Z\}$ forms a module over the Virasoro algebra on~$V$. (The definition of the Virasoro algebra is given in~\S\ref{s.minimal_models}.) 
The action of the central element~$C$ is required to be a complex number, which is denoted by~$c_V\in\C$ called the \textit{central charge} of~$V$ in the literature.  
For each~nonnegative integer~$n$ any element of a homogeneous subspace~$V_n$ has weight~$n$ with respect to~$L_0$, i.e.,~$L_0v=nv$ for all~$v\in V_n$.  
In this paper we always suppose that~$V_0=\C\cdot\unit$, and then~$V$ is called of \textit{CFT type}. Let~$V$ be a VOA with central charge~$c_V$. 
A \textit{weak~$V$-module} is a pair~$(M,Y_M)$ of a vector space~$M$ and a linear map~ $Y_M:V\to\End_\C(M)[[z, z^{-1}]]$ 
with required natural conditions as a module on~$V$. A weak $V$-module~$M$ is also a module over the Virasoro algebra 
with the same central charge~$c_V\in\C$ by letting~$Y_M(\omega,z)=\sum_{n\in \Z}L_nz^{-n-2}$. (For more details, see~\S3 in~\cite{LL} and~\S4 in~\cite{MN}.)

A weak $V$-module is said to be a \textit{$V$-module} if there are finitely many complex numbers~$h_i$ 
such that~$M=\bigoplus_i\bigoplus_{n=0}^\infty M_{h_i+n}$, where~$M_h$ is a \textit{finite-dimensional} eigenspace of~$L_0$ with an eigenvalue~$h$. 
If~$M$ is a simple $V$-module (in the obvious meaning), then there exists a complex number~$h$ 
such that~$M=\bigoplus_{n=0}^\infty M_{h+n}$ with~$M_h\neq0$. This complex number~$h$ is called the \textit{conformal weight} 
of the simple $V$-module~$M$. Then the (formal) \textit{character} of the simple $V$-module~$M$ is defined by 
\begin{equation*}
\ch_M(\tau)\=\tr_Mq^{L_0-\frac{c_V}{24}}\=q^{h-\frac{c_V}{24}}\sum_{n=0}^\infty(\dim  M_{h+n})\,q^n\quad(\,q\=e^{2\pi i\tau}\,)\,.
\end{equation*}

\begin{definition}
(a)~We say that a vertex operator algebra~$V$ is \textit{rational} if any $V$-module is completely reducible, and is \textit{of finite type} 
if the codimension of the subspace of~$V$, which is spanned by elements~$a_{(-2)}\,b$ for all~$a,\,b\in V$,  is finite. 
It is known that conformal weights and central charge of such a~$V$ are rational numbers. (cf.~\cite{AM}).\\
(b)~We denote by~$\sch_V$ the space spanned by characters of all (inequivalent) simple  $V$-modules.
\end{definition}
\noindent
It was shown in Theorem~4.4.1 of~\cite{Zhu} that if~$V$ is of finite type, then the number of simple $V$-modules is finite 
(up to isomorphism), and the character~$\ch_M$ of any simple $V$-module~$M$ converges in the complex upper half-plane~$\H$. 
Moreover, if rational VOA~$V$ is of CFT and of finite type, then the space~$\sch_V$ is invariant under the slash~0 action~$|_0$ of~$\MG$. 
This fact is now called the \textit{modular invariance property} of the space~$\sch_V$. 

\section{Minimal models}\label{s.minimal_models} 
The minimal models is one of the important examples of rational simple VOAs of CFT and finite type. The Virasoro algebra denoted here by~$\Vir$ is a Lie algebra 
spanned by~$L_n$~$(n\in\Z)$ and a central element~$C$ with commutation relations
\begin{equation*}
[L_m, L_n\,]\=(m-n)L_{m+n}+\frac{m^3-m}{12}\,\delta_{m,\,-n}\,C\quad (m,\,n\in\Z)\,.
\end{equation*}
Let~$\Vir_{\geq0}=\bigoplus_{n=0}^\infty\C L_n\oplus\C\, C$, and let~$c$ and~$h$ be complex numbers. We denote by~$\C\,v_{c,\,h}$ 
the one-dimensional $U(\Vir_{\geq0})$-module that is defined by~$L_n v_{c,\,h}=h\,\delta_{n,0}v_{c,\,h}$ and~$C\,v_{c,\,h}=c\,v_{c,\,h}$, 
where~$U(\Vir)$ is the universal enveloping algebra of the Lie algebra~$\Vir$. Then the induced module~$M(c,h)=U(\Vir)\otimes_{U(\Vir_{\geq0})}\C\,v_{c,\,h}$ 
is called the \textit{Verma module} of the Virasoro algebra with central charge~$c$ and the highest weight~$h$. 
If Verma module has the \textit{unique} maximal proper submodule~$J(c,h)$, then we denote its simple quotient by~$L(c,h)=M(c,h)/J(c,h)$. 
Let~$\langle L_{-1}\,v_{c,\,0}\rangle$ be a left $\Vir$-submodule of~$M(c,0)$ which is generated by~$L_{-1}v_{c,\,0}$.
One can verify that the space~$\langle\, L_{-1}\,v_{c,\,0}\,\rangle$ is a proper submodule of~$M(c, 0)$. We denote by~$V(c,0)$ 
the quotient $\Vir$-module~$M(c,0)/\langle\,L_{-1}\,v_{c,\,0}\,\rangle$, and by~$\unit$ the image of~$v_{c,\, 0}$ under the natural surjection. It was proved in Theorem~4.3 (p.~162) and Remark (p.~163) of~\cite{FZ} that~$V(c,0)$ and~$L(c,0)$ are VOAs with the vacuum vector~$\unit$ and the Virasoro vector~$\omega=L_{-2}\unit$, respectively. 
Then by applying the Poincar\'{e}-Birkhoff-Witt theorem to~$U(\Vir)$ one can find that the space~$V(c,0)$ has a basis which consists 
of~$\unit$ and~$L_{-n_1}\dotsm L_{-n_r}\,\unit$~$(r>0\,,n_i\in \Z\,,n_1\geq\dotsb\geq n_r>1)$, whose weights induced by~$L_0$ are~0 and~$\sum_{i=1}^rn_i$, respectively. 
Therefore the character of $\Vir$-module~$V(c,0)$ is given by
\begin{equation*}
\ch_{V(c,\,h)}(\tau)\=q^{-c/24}\prod_{n>1}(1-q^n)^{-1}\=q^{-c/24}\left(1+q^2+q^3+2q^4+\cdots\right)\,.
\end{equation*}

A nonzero vector~$v$ of a (highest weight) $\Vir$-module is called \textit{singular} if~$L_n\,v=0$ for all positive integers~$n$. 
If~$V(c,0)$ contains a singular vector~$v$ with a \textit{positive weight}, then~$V(c, 0)$ is not simple since~$U(\Vir)v$ 
is a proper vertex operator subalgebra of~$V$. Using the Kac determinant formula (cf.~\cite{K}), it was shown in~\cite{FeF1} and~\cite{FeF2} 
that~$V(c,0)$ contains a (positive weight) singular vector if and only if its central charge is given by
\begin{equation}\label{E.Cpq}
c\=c_{\,P,\,Q}\=1-\frac{6(P-Q)^2}{PQ}
\end{equation}
with coprime positive integers~$P$ and~$Q$. It was moreover proved in~\cite{FeF2},~\cite{FeF1} and in Theorem~6.5 in~\cite{IK}  
that the maximal proper submodule~$J(c_{P,\,Q},0)$ of~$M(c_{P,\,Q},0)$ is generated by two singular vectors,~$L_{-1}v_{\,c_{P,\,Q}}$
and~$w_{P,\,Q}$, where the latter has weight~$(P-1)(Q-1)$. We denote by~$L(c_{P,\,Q},0)$ the simple Virasoro VOA~$M(c_{P,\,Q},0)/J(c_{P,\,Q},0)$ 
which is called the \textit{minimal model} of type~$(P,Q)$. Let~$P$ and~$Q$ be coprime integers and define rational numbers by
\begin{equation}\label{eq.highest_weight}
h_{\,P,\,Q\,;\,r,\,s}\=\frac{(rQ-sP)^2-(P-Q)^2}{4PQ}\,.
\end{equation}
\begin{notations}
We denote~$L(c_{P,\,Q},0)$ and~$L(c_{P,\,Q},h_{\,P,\,Q\,;\,r,\,s})$ by~$L_{P,\,Q}$ and~$L_{P,\,Q\,;\,r,\,s}$, respectively.  
In particular, one can see~$L_{P,\,Q}=L_{P,\,Q\,;\,1,\,1}$.
\end{notations}

The following facts are well known (cf.~Theorem~6.13 in~\cite{IK}).
\begin{theorem}\label{T.Minimal_Main}
Let~$P$ and~$Q$ be coprime integers with~$0<P<Q$.\\
\textup{(a)}
The vertex operator algebra~$L_{P,\,Q}$ is rational, simple and is of CFT and finite type.\\
\textup{(b)}
For any simple~$L_{P,\,Q}$-module~$M$ there exist integers~$0<r<P$ and~$0<s<Q$ such that~$M$ is isomorphic to~$L_{P,\,Q\,;\,r,\,s}$.\\
\textup{(c)}
The~$L_{P,\,Q}$-modules~$L(c_{P,\,Q},h_1)$ and~$L(c_{P,\,Q},h_2)$ are isomorphic if and only if~$h_1=h_2$.\\
\textup{(d)}
The rational number~$h_{\,P,\,Q\,;\,r,\,s}$ is the conformal weight of the $L_{P,\,Q}$-module~$L_{P,\,Q\,;\,r,\,s}$.
\end{theorem}
\noindent
Furthermore, it is known that the character~$\ch_{\,P,\,Q\,;\,r,\,s}$ is written in the form
\begin{equation}\label{eq.simple_character}
\ch_{\,P,\,Q\,;\,r,\,s}(\tau)\=\frac{1}{\eta(\tau)}\bigl(\theta_{PQ,\, Qr-Ps}(\tau)-\theta_{PQ,\, Qr+Ps}(\tau)\bigr)\,,\quad
\theta_{a,\,b}(\tau)\=\sum_{n\in\Z}q^{a\left(n+b/2a\right)^2}
\end{equation}
for any half-integers~$a$ and~$b$. (See Theorem~6.13 of~\cite{IK}.) This set of characters coincides with the solution space of some monic linear differential equation
and of order~$\frac{(P-1)(Q-1)}{2}$. (See~Theorem~6.1 of~\cite{Milas1}, Theorem~6.1 of~\cite{Milas} and~Theorem~2.1 of~\cite{MMO}) 

Let~$P$ and~$Q$ are coprime positive integers greater than one. Let~$r$ and~$s$ be integers with~$1\leq r\leq P-1$ and~$1\leq s\leq Q-1$. 
Then central charge and conformal weights of the minimal model of type~$(P, Q)$ can be expressed as
\begin{equation*}
c\=c_{P,Q}\=1-\frac{6(P-Q)^2}{PQ}\,,\qquad h_{r,s}\=\frac{(rQ-Ps)^2-(P-Q)^2}{4PQ}\,.
\end{equation*}
We now compute the \textit{effective central charge} and the \textit{the minimal conformal weight} of  the minimal model of type $(P,Q)$. 
\begin{lemma}
Let~$P$ and~$Q$ be coprime positive integers greater than one, and let~$r$ and~$s$ be integers with~$1\leq r\leq P-1$ and~$1\leq s\leq Q-1$. 
Then the minimum~$h_{\min}^{P,Q}$ of the set of  all conformal weights~$\{h^{P,Q}_{r,s}\}$ and effective central charge~$\ec^{P,Q}$  are given by
\begin{equation}
h^{P,Q}_{\min}\=\frac{1-(P-Q)^2}{4PQ}\,\quad\text{and}\quad\ec^{P,\,Q}\=\frac{PQ-6}{PQ}\,.
\end{equation}\label{eq.hc}
\end{lemma}
\begin{proof}
Since~$P$ and~$Q$ are coprime, we have~$r_1P-Qs_1=1$ for some integers~$r_1$ and~$s_1$. Then any solution~$(r,s)\in\Z^2$ 
of the Diophantine equation~$Pr-Qs=1$ is written as~$r=r_1+Qk$,~$s=s_1+Pk$ with an integer~$k$. 
Obviously~$(r_1,P)=1$ and~$(s_1,Q)=1$, and then  we can take~$r$ with~$1\leq r\leq P-1$ for some integer~$k$  since~$(r,P)=1$ and~$P>1$.
Hence we have~$Q-1\leq rQ-Q=Ps\leq Q(P-1)-1$ and~$0<(Q-1)/P\leq s\leq Q-(Q+1)/Q\leq Q-1$. Thus we have found~$r$ and~$s$ so that~$Qr-Ps=1$ 
and~$1\leq r\leq P-1$ and~$1\leq s\leq Q-1$. Therefore the minimum value of~$h^{P,Q}_{r,s}$ is given by the first of  equation~\eqref{eq.hc}.
Then the formula of the effective central charge immediately follows by the definition.
\end{proof}

We now obtain
\begin{corollary}
Let~$\ec$ and~$h_{\min}$ be the effective central charge and the minimal conformal weight of the minimal model of type~$(2,p)$. Then we have
\begin{equation*}
h_{\min}\=\frac{1-(2-p)^2}{8p}\quad\text{and}\quad\ec\=1-\frac{3}{p}\,,
\end{equation*}
where~$h_{\min}=h^{2,p}_{\min}$ and~$\ec=\ec^{2,p}$, respectively.
\end{corollary}

\section{Minimal models and Ibukiyama modular forms}\label{s.minimal_Ibukiyama}
Now, we can relate the Ibukiyama modular forms to characters of the minimal models of type~$(2,p)$.
We show that any character of the minimal model of type~$(2,p)$ multiplied by~$\eta^{\ec}$ is equal  to  an Ibukiyama modular form up to a multiple constant.

\begin{theorem}\label{t.minimal_Ibukiyama}
Let~$L_{2,\,p}$ be the minimal model of type~$(2, p)$ with an odd integer~$p\,(>3)$. Then we have
\begin{equation}\label{eq.mi}
\eta(\tau)^{\ec}\ch_{\,2,\,p\, ; \,1,\,s}(\tau)\=e\left(\frac{2ps-1+p-p^2}{8p}\right)f_{p-2s}(\tau)\,,\quad\ec\=1-\frac{3}{p}\,.
\end{equation}
The set of~$\{p-2s\mid 0<s<p/2\}$ equals the set~$\{1\leq r\leq p-2\mid\text{$r$ is odd\,}\}$, and  where~$f_r$ 
\textup{(}the index~$r$ is any element in latter set\thin\textup{)} 
is the Ibukiyama modular form given in~\textup{Theorem~\ref{t.im}}. 
\end{theorem}
\begin{proof}
Using equation~\eqref{eq.simple_character} one can compute the left-hand side of the expected identity~\eqref{eq.mi} as
\begin{equation*} 
\begin{split}
&\eta(\tau)^{\ec}\ch_{\,2,\,p\,;\,1,\,s}(\tau)
\=\eta(\tau)^{-3/p}\sum_{n\in\Z}\left[e(\tau)^{\frac{p}{2}\left(\frac{2n}{2}+\frac{p-2s}{2p}\right)^2}
-e(\tau)^{\frac{p}{2}\left(\frac{2n-1}{2}+\frac{p-2s}{2p}\right)^2}\right]\\
&\=\eta(\tau)^{-3/p}\sum_{n\in\Z}(-1)^ne(\tau)^{\frac{p}{2}\left(n+\frac{p-2s}{2p}\right)^2}\=e\Bigl(\frac{2s-p}{4p}\Bigr)\eta(\tau)^{-3/p}\Theta_{\left(\frac{p-2s}{2p}\,,\,
\frac{1}{2}\right)}(p\thin\tau, 0)\\
&\=e\Bigl(\frac{2ps-1+p-p^2}{4p}\Bigr)f_{p-2s}(\tau)\,,
\end{split}
\end{equation*}
and hence equation~\eqref{eq.mi} holds.
\end{proof}

We now give a different proof of Theorem~\ref{t.im}~(b), i.e. Theorem~1.2 in~\cite{I}. Using Proposition~6.3 in~\cite{IK} 
one can see that the action of~$S=\begin{psmallmatrix}0&-1\\1&0\end{psmallmatrix}$ on the character of~$L_{\,2,\,p\,;\,1,\,s}$ is given by 
\begin{equation*}
\ch_{\,2,\,p\,;\,1,\,s}(-1/\tau)
\=\frac{2}{\sqrt{p}}\sum_{r=1}^{\frac{p-1}{2}}(-1)^{(1+r)(1+s)}\sin\Bigl(\frac{\pi}{2}(2-p)\Bigr)\sin\Bigl(\frac{\pi s\thin r}{p}(2-p)\Bigr)\ch_{\,2,\,p\,;\,1,\,r}(\tau)\,.
\end{equation*}
Therefore we have 
\begin{equation*}
\begin{split}
\pi_p(S)f_{p-2s}
&\=e\Big(\frac{3p-1}{8}\Big)\frac{2}{\sqrt{p}}\sum_{r=1}^{\frac{p-1}{2}}\cos\Big(2\pi\Big(\frac{p}{4}+\frac{rs}{p}-s\Big)\Big)f_{p-2r}\\
&\=e\Big(\frac{3p-1}{8}\Big)\frac{1}{\sqrt{p}}\sum_{r=1}^{\frac{p-1}{2}}\Big[e\Big(-s+\frac{p}{4}+\frac{rs}{p}\Big)+e\Big(s-\frac{p}{4}-\frac{rs}{p}\Big)\Big]f_{p-2r}\\
&\=e\Big(\frac{3p-1}{8}\Big)\frac{1}{\sqrt{p}}\sum_{r=1}^{\frac{p-1}{2}}\Big[e\Big(\frac{(p-2s)-(p-2r)}{4}+\frac{(p-2s)(p-2r)}{4p}\Big)\\
&\qquad\qquad\qquad\qquad+e\Big(-\frac{(p-2s)-(p-2r)}{4}-\frac{(p-2s)(p-2r)}{4p}\Big)\Big]f_{p-2r}\,.
\end{split}
\end{equation*}
Hence it follows for any odd integer~$s$ with~$1\leq s\leq p-2$ that  
\begin{equation*}
\begin{split}
\pi_p(S)f_{s}&\=e\left(\frac{3p-1}{8}\right)\frac{1}{\sqrt{p}}\sum_{\substack{1\leq r\leq p-2\,,\\ r\equiv1\mod{2}}}\bigg[e\left(\frac{s-r}{4}+\frac{sr}{4p}\right)
+e\left(-\frac{s-r}{4}-\frac{sr}{4p}\right)\bigg]f_{r}\\
&\=\frac{1}{\sqrt{p}}\sum_{\substack{1\leq r\leq p-2\,,\\ r\equiv1\mod{2}}}\bigg[e\left(\frac{s-r}{4}+\frac{sr}{4p}+\frac{3p-1}{8}\right)+(-1)^{\frac{p+1}{2}}e\left(-\frac{s-r}{4}
-\frac{sr}{4p}-\frac{3p-1}{8}\right)\bigg]f_{r}\,.
\end{split}
\end{equation*}
From the definition of $\pi_p$ we obtain~$\pi_{p}(T)f_{s}=e\left(\frac{1-p^2}{8p}\right)e\bigl(\frac{s^2-1}{8p}\bigr)f_{s}(\tau)=e\bigl(\frac{s^2-p^2}{8p}\bigr)f_{s}(\tau)$.
Therefore we have proved  Theorem~\ref{t.im}~(b).

The characters of a VOA are not always modular functions on~$\MG$. However, it was proved by Dong,~Lin and~Ng that they are modular functions
on some principal congruence subgroups of~$\MG$. 

\begin{theorem}[Theorem~I and Theorem~3.10 in \cite{DLN}]\label{T.DLX}
Let~$V$ be a rational simple self-dual vertex operator algebra of CFT and of finite type, and let~$\{M^i\}$ be a complete list of simple $V$-modules up to isomorphism. 
Then each character~$\chi_{M^i}$ is a modular function on the principal congruence subgroup of~$\MG$ of level~$n$ that is the smallest positive integer such 
that~$n(\lambda_i-c/24)$ is an integer for all~$i$, where~$\lambda_i$ is the conformal weight of~$M^i$.
\end{theorem}

Let $m$ be a positive integer with~$m|n$. The character~$\chi_{M^i}$ is a modular function on~$\G(n)$ but is not always modular on the group~$\G(m)$. We denote~$n$ in Theorem~\ref{T.DLX} by~$n_p$ for the minimal models of type~$(2,p)$. For instance, we have~$n_5=60$, however, the Ibukiyama modular form 
(given by~\eqref{eq.mi}) is defined 
on~$\G(5)$. Therefore the~$\chi_{M^i}(=\ch_{M^i}$) must be a modular function with a certain multiplier system on~$\G(5)$. 

We now give a formula  of~$n_p$ of the minimal model of type~$(2,p)$.  Any character of minimal models of type~$(2, p)$ with an odd integer~$p>3$ has an $q$-expansion
\begin{equation}\label{eq.q-expansion}
\frac{1}{\eta(q)}\sum_{n\in\mathbb{Z}}(-1)^{n}q^{\frac{p}{2}\left(n+\frac{p-2s}{2p}\right)^{2}}
\=q^{\frac{(p-2s)^2}{8p}-\frac{1}{24}}+O\Big(q^{\frac{(p-2s)^2}{8p}-\frac{1}{24}+1}\Big)
\end{equation}
for $s=1,\,2,\,\ldots,\,(p-1)/2$. We will find the smallest positive integer~$n_p$ such that $n\bigl(\frac{(p-2s)^2}{8p}-\frac{1}{24}\bigr)=n\bigl(\frac{3(p-2s)^2-p}{24p}\bigr)$ 
is an integer for all~$s$.
We now let~$p=2r+3$ with~$r>0$. Then we have
\begin{theorem}\label{t.n}
Let $p=2r+3$ with a positive integer~$r$. Then the minimal  denominator of the irreducible fraction of~$\frac{3(p-2s)^2-p}{24p}$ for any~$0\leq s\leq(p-1)/2$ is equal to
\begin{equation}\label{eq.n}
\frac{12(2r+3)}{\gcd(4,r)\gcd(3,r)}\,.
\end{equation}
\end{theorem}
\begin{proof}
Let~$r=12m+\ell$, where~$0\leq\ell\leq11$ and~$m$ is a nonnegative integer. We will give the proof in the case $\ell=0$. The cases~$1\leq\ell\leq11$ 
are described in the appendix and they are  very similar to the case of~$\ell=0$.

Now, on the one hand~$\frac{12(2r+3)}{\gcd(4,r)\gcd(3,r)}=24m+3$, on the other hand, 
\begin{equation*}
\frac{3(p-2s)^2-p}{24p}\=\frac{3(24m+3-2s)^2-24m-3}{24(24m+3)}\=\frac{(24m+3-2s)^2-8m-1}{8(24m+3)}\,.
\end{equation*}
Since~$(24m+3-2s)^2-8m-1\equiv0\pmod{8}$, we see that the denominator of irreducible fraction~$\frac{3(p-2s)^2-p}{24p}$ is equal to~$24m+3$.
Other cases for $\ell=1, \ldots, 11$ are proved in the same way as mentioned in the appendix.

\end{proof}

\begin{corollary}
Let $n_p$ be the integer in Theorem~\ref{T.DLX} of the minimal model of type~$(2,p)$. Then it is equal to~$\frac{12(2r+3)}{\gcd(4,r)\gcd(3,r)}$.
\end{corollary}

Let $p=2r+3$ with a positive integer~$r$.  The following Table~\ref{table:n} is a part of the list of the denominator of~irreducible fraction 
of~$\frac{12(2r+3)}{\text{gcd}(4,r)\text{gcd}(3,r)}$ 
for~$r\equiv\ell\pmod{12}$.

\begin{table}\label{table:n}
\begin{center}
\begin{tabular}{c|c|c|c|c|c|c|c|c|c|c|c|c}
$\ell$&$0$&$1$&$2$&$3$&$4$&$5$&$6$&$7$&$8$&$9$&$10$&$11$\\ \hline
$n$&$p$&$12p$&$6p$&$4p$&$12p$&$12p$&$4p$&$12p$&$12p$&$4p$&$6p$&$12p$\\
\end{tabular}
\end{center}
\caption{Table of $n_p$}
\end{table}

\section{Modular linear differential equations and vector-valued modular forms}\label{MLDE} 
In this section we recall the basic fact that the space spanned by characters of all simple modules over a minimal model equals
the solution space of some monic MLDE whose definition is given below. Let~$k$ be a positive integer and let~$E_{2k}$ be 
the \textit{normalized Eisenstein series} of weight~$2k$, which is defined 
by~$E_{2k}(\tau)\=1-(4k/B_{2k})\sum_{n=1}^\infty\sigma_{2k-1}(n)q^n$ and~$q\=e^{2\pi i\tau}$ $(\tau\in\H)$, where~$\sigma_j(n)$ 
for~$n\in\N$ is the sum of the $j$\thin th powers of the positive divisors of~$n$, and where~$B_n$ is the $n$\thin th Bernoulli number.

Let~$\sd:\Hol(\H)\to\Hol(\H)$ be an \textit{operator} that is defined by~$\sd_k(f)\=D(f)-\frac{k}{12}E_2(f)$ for each rational number~$k$, 
where~$D=(2\pi i)^{-1}d/d\tau$ 
for any~$f\in M_k(\MG)$, which is called the \textit{Serre operator}. It is easy to show that~$\sd:M_*(\MG)\to M_{*+2}(\MG)$. 
The $i$\thin th \textit{iterated Serre operator} is defined by~$\sd_k^{i}=\sd_{k+2(i-1)}\circ\dotsb\circ\sd_{k+2}\circ\sd_k$ with~$\sd_k^0=1$, 
any differential equation of the~$n$\thin th order is written as
\begin{equation}\label{eq.mlde_Serre}
\sd_k^n (f)+\sum_{i=0}^{n-1}P_i\sd_k^i (f)\=0
\end{equation}
is called a \textit{monic modular linear differential equation} (monic MLDE) of weight~$k$ if~$P_i$ is a holomorphic modular form 
of weight~$2(n-i)$ for any~$i$. Since any holomorphic modular form of weight two in~$\H$ is~0, the coefficient~$P_{n-1}$ 
of~\eqref{eq.mlde_Serre} is always zero.
It was shown by Mason (Theorem~4.1 in~\cite{Mason}) that the solution space of any monic MLDE of weight~$k$ is closed under the usual 
slash~$k$ action~$|_k$ of~$\MG$.

\begin{definition}
(a) Let~$k$ be a rational number. A \textit{vector-valued modular form} (VVMF) of weight~$k$ on~$SL(2,\Z)$ with a multiplier system~$v:SL(2,\Z)\to\C$ 
with~$v(-I)=(-1)^{k}$ is a column vector-valued holomorphic function~$\VF={}^t\left(f_1,\,f_2,\,\dotsc,\,f_n\right)$ 
in the upper half-plane~$\H$ equipped with an~$n$-dimensional representation~$\pi:SL(2,\Z)\to GL_n(\C)$ that is given 
by~$j(\gamma,\tau)^{-1}\,\VF(\gamma(\tau))=\pi(\gamma)\VF(\tau)$ 
for all~$\gamma=\left(\begin{smallmatrix}a&b\\c&d\end{smallmatrix}\right)\in SL(2,\Z)$ with~$j(\gamma, \tau)=v(\g)(c\tau+d)^k$, 
and each component function~$f_j$ has a Fourier expansion $f_j=q^{\lambda_j}\bigl(\sum_{i=0}^\infty a_i^jq^i\bigr)$ in~$\H$, 
with~$a_0^j\neq0$ and~$\lambda_j\in\Q$.\\
(b) A vector-valued modular form~$\VF$ is called \textit{normalized} if~$\lambda_1>\lambda_2>\dotsb>\lambda_r$ 
and~$f_j=0$~$(r+1\leq j\leq n)$. If the component functions of~$\VF$ are linearly independent, then there is an invertible matrix~$A$ 
such that~$A\VF$ is normalized. In particular, if~$\lambda_1,\,\dotsc,\, \lambda_n$ are mutually distinct, then they are automatically normalized after consecutive interchanges of indices. The VVMF~$\VF$ is called \textit{strictly normalized} when~$r=n$.
\end{definition}

 We have the following fact which was also proved by G.~Mason (Theorem~4.3 in~\cite{Mason}).
 Very recently we have proved a similar result for the theta gamma group in~\cite{ANS}.

\begin{proposition}\label{t.exponents}
The space spanned by all component functions~$\{f_j\}_{j=1}^n$ of a strictly normalized vector-valued modular form of weight~$k$ 
is equal to the solution space of some monic modular linear differential equation of the $n$\thin th order 
if and only if~$n(n+k-1)-12\sum_{i=1}^n\lambda_j\=0$, where each~$\lambda_j$ \textup{($1\leq j\leq n$)} is the leading exponent of~$f_j$.
\end{proposition}

\noindent
Since any monic MLDE has a unique regular singular point at~$q=0$ one can use the Frobenius method described in Chapter XVI in~\cite{Ince} 
to obtain a basis of the solution space of a monic MLDE. Suppose that~$f=q^{\,\lambda}\left(1+\sum_{i=0}^\infty a_iq^i\right)$ is a solution of equation~\eqref{eq.mlde_Serre}. Then, substituting~$f$ into~\eqref{eq.mlde_Serre} and taking the coefficient of the lowest power term, 
one can see that~$\lambda$ is a root of an algebraic equation (called an \textit{indicial equation})~$\Psi(t)=0$,
 where~$\Psi(t)=\prod_{\ell=0}^{n-1}\left(t-\frac{\ell}{6}\right)+\sum_{i=0}^{n-2}P_{i,0}\prod_{\ell=0}^{i-1}\left(t-\frac{\ell}{6}\right)$ 
and~$P_i\=\sum_{n=0}^\infty P_{i,n}\,q^n$. This polynomial is called the \textit{indicial polynomial} of the monic MLDE~\eqref{eq.mlde_Serre} 
and the roots of the indicial equation are called \textit{indicial roots}. 

\begin{lemma}\label{L.Frobenius}
If any pair of two different indicial roots have a non-integral difference, then there exists a unique solution 
of the form~$f=q^{\lambda}\left(1+\sum_{n=1}^\infty a_n q^n\right)$, 
where~$\lambda$ is any indicial root and the coefficients~$a_n$~$(n\in\Z_{\geq0})$ are uniquely determined from~$a_0$ by some recursive relations 
and indicial roots. 
If there is a double incidicial root, then there always exists a logarithmic solution that is a polynomial of~$\tau$ with coefficients
 in~$q^\lambda\,\C((q))$ with a complex number~$\lambda$.
\end{lemma}
It is important to know that the modular linear differential equation associated to a rational simple vertex operator algebra has 
\textit{no logarithmic solutions} even if two different indicial roots have an integral difference.

\section{Some properties of the Ibukiyama modular forms }

We have established some relations between the Imukiyama modular forms and the characters of the minimal models of type~$(2,p)$.
It is therefore natural to ask whether we can obtain modular forms of fractional weights from the characters of \textit{any} minimal model 
of type~$(P,Q)$. The characters of the minimal model of type~$(P,Q)$ is a modular function on~$\G(n)$, where~$n$ is defined in Theorem~\ref{T.DLX}.
Of course, we can define fractional modular forms~$\eta^{\ec}\ch_M$ on~$\G(n)$ with the multiplier system 
which is  the same as that of~$\eta^{\ec}$ and then  $\eta^{\ec}\ch_M$ is standard on~$\G(n)$.
However, it is not clear to us whether $\eta^{\ec}\ch_M$ is standard on a smaller congruence group~$\G(m)$.
Moreover, we want to know whether the following property holds for the Ibukiyama modular forms

\mn
(A)~Let $m$ and~$n$ be positive integers with~$m|n$. Then~$M_k(\G(m), \nu_m)\subset M_k(\G(n),\mu_n)$ for~$k\in\Q$ 
with~$\mu_n|\G(m)=\nu_m$, where $\nu$ and $\mu$ are multiplier systems of $\G(m)$ and  $\G(n)$.

\medskip 
We will give four examples of spaces of the Ibukiyama modular forms of fractional weights, and show 
that the set of pairs~$(m,n)=(5,15)$ and~$(7,21)$ satisfies (A).  Since the space of Ibukiyama modular forms~$M_{\frac{p-3}{2p}}(\G(p), v_p)$ is invariant by the $\MG$-action, it may be equal to the solution space of of some MLDE. In fact, we will give the corresponding monic MLDE. For any positive odd integer~$p>3$ we set 
\begin{equation*}
\f_s(\tau):\=e\left(\frac{2ps-1+p-p^2}{4p}\right)f_{p-2s}\quad\text{for}\quad1\leq s\leq\frac{p-1}{2}\,.
\end{equation*}
First, it was shown in Lemma~1.7 of~\cite{I} that~$\dim M_{\frac{p-3}{2p}}(\G(p), v_p)>1$ with any odd integer~$p>3$.
In particular,~$\dim M_{\frac{p-3}{2p}}(\G(p), v_p)=2$ and~3 for~$p=5$ and~7, respectively. (See p.~325, lemma1.7 of Section~1.3 in~\cite{I}) 
To show (A) 
for the pairs~$(m,n)=(5,15)$ and~$(7,21)$ it is, in principle, necessary to know basis of~$M_{\frac{1}{5}}(\G(5), v_5)$ and~$M_{2/7}(\G(7),v_7)$. 
Each basis of the space~$M_{\frac{1}{5}}(\G(5), v_5)$ and~$M_{2/7}(\G(7),v_7)$ is known (pp.~326--332, Section 2.1--2.3 in~\cite{I}) 
(however, in the cases~$p=15, 21$  they are not known);  the former basis  consists of
\begin{align}\label{E.f12}
\f_1^{[5]}(\tau)&:\=\eta(\tau)^{2/5}\ch_{2,\,5;1,1}(\tau)\=\eta(\tau)^{-3/5}\sum_{n\in\Z}(-1)^ne(\tau)^{\frac{5}{2}\left(n+\frac{3}{10}\right)^2}\,,\\
\f^{[5]}_2(\tau)&:\=\eta(\tau)^{2/5}\ch_{2,\,5;1,2}(\tau)\=\eta(\tau)^{-3/5}\sum_{n\in\Z}(-1)^ne(\tau)^{\frac{5}{2}\left(n+\frac{1}{10}\right)^2}\,,\label{E.f13a}
\end{align}
and the associated monic MLDE
is~$\sd^2_{1/5}(f)-\frac{11}{3600}E_4(f)=0$. 
The latter has the effective central charge of~$4/7$. Then the Ibukiyama modular forms 
are~$\f_1^{[7]}(\tau):=\eta(\tau)^{4/7}\ch_{2,7;1,1}(\tau)$, $\f_2^{[7]}:=\eta(\tau)^{4/7}\ch_{2,7;1,2}(\tau)$ 
and~$\f_3^{[7]}:=\eta(\tau)^{4/7}\ch_{2,7;1,3}(\tau)$. More precisely, they are written 
as~$\f_j^{[7]}(\tau)\=\eta(\tau)^{-3/7}\sum_{n\in\Z}(-1)^n\thin e(\tau)^{\frac{7}{2}\left(n+\frac{7-2s}{14}\right)^2}$ for~$1\leq j\leq3$,
which form a basis of the space~$M_{2/7}(\G(7),v_7)$, and then the $q$-expansion of~$\f_j^{[7]}$ for~$1\leq j\leq3$ is
given by 
\begin{equation}\label{eq.q-series}
\begin{aligned}
\f_1^{[7]}(\tau)&\=q^{3/7}\big(1-\frac{4}{7}q+\frac{15}{49}q^2+\frac{43}{343}q^3+\frac{965}{2401}q^4
+\frac{1524}{16807}q^5-\frac{39026}{117649}q^6-\frac{1874288}{5764801}q^7+\cdots\big)\,,\\
\f_2^{[7]}(\tau)&\=q^{1/7}\big(1+\frac{3}{7}q-\frac{13}{49}q^2+\frac{148}{343}q^3+\frac{1266}{2401}q^4
-\frac{8528}{16807}q^5+\frac{38870}{117649}q^6+\frac{213504}{5764801}q^7+\cdots\big)\,,\\
\f_3^{[7]}(\tau)&\=1+\frac{3}{7}q+\frac{36}{49}q^2-\frac{48}{343}q^3-\frac{400}{2401}q^4+\frac{3183}{16807}q^5
+\frac{50140}{117649}q^6+\frac{13535}{5764801}q^7+\cdots\,.
\end{aligned}
\end{equation}
The associated monic MLDE is~$\sd^3_{2/7}(f)-\frac{5}{252}E_4\sd_{2/7}(f)+\frac{85}{74088}E_6(f)\=0$.

We next treat fractional modular forms associated to the minimal model of type~$(2,15)$.  Since the effective central charge of this case is~$4/5$
the Ibukiyama modular forms, which do not  form a basis of $M(\G(15),v_{15}$), are written as  
\begin{alignat*}{5}
&\f_1^{[15]}(\tau)\=\eta(\tau)^{4/5}\ch_{2,15;1,1}(\tau)\,,&\quad&\f^{[15]}_2(\tau)\=\eta(\tau)^{4/5}\ch_{2,15;1,2}(\tau)\,,
&\quad&\f^{[15]}_3(\tau)\=\eta(\tau)^{4/5}\ch_{2,15;1,3}(\tau)\,,\\
&\f_4^{[15]}(\tau)\=\eta(\tau)^{4/5}\ch_{2,15;1,4}(\tau)\,, &\quad&\f^{[15]}_5(\tau)\=\eta(\tau)^{4/5}\ch_{2,15;1,5}(\tau)\,,
&\quad&\f^{[15]}_6(\tau)\=\eta(\tau)^{4/5}\ch_{2,15;1,6}(\tau)\,,\\
&\f_7^{[15]}(\tau)\=\eta(\tau)^{4/5}\ch_{2,15;1,7}(\tau)\,.
\end{alignat*}
It is not difficult to see that~$\f^{[15]}_j(\tau)\=\eta(\tau)^{-1/5}\sum_{n\in\Z}(-1)^ne(\tau)^{\frac{15}{2}\left(n+\frac{2j-1}{30}\right)^2}$ 
for any~$1\leq j\leq7$ and their Fourier expansions are given by
\begin{equation}\label{eq.p=15}
\begin{aligned}
&\f^{[15]}_1(\tau)\=q^{7/5}\left(1-\frac{4}{5}q+\frac{3}{25}q^2+\frac{1}{125}q^3+\frac{79}{625}q^4
-\frac{776}{15625}q^5+\frac{16609}{78125}q^6+\cdots\right)\,,\\
&\f^{[15]}_2(\tau)\=q\left(1+ \frac{1}{5}q- \frac{17}{25}q^2+\frac{16}{125}q^3+\frac{84}{625}q^4+\frac{1199}{15625}q^5+\cdots\right)\,,\\
&\f^{[15]}_3(\tau)\=q^{2/3}\left(1 + \frac{1}{5}q+\frac{8}{25}q^2-\frac{84}{125}q^3+\frac{159}{625}q^4
+\frac{1324}{15625}q^5+\frac{22604}{78125}q^6+\cdots\right)\,,\\
&\f^{[15]}_4(\tau)\=q^{2/5}\left(1 + \frac{1}{5}q+\frac{8}{25}q^2+\frac{41}{125}q^3
-\frac{341}{625}q^4+\frac{3199}{15625}q^5+\frac{23229}{78125}q^6+\cdots\right)\,,\\
&\f^{[15]}_5(\tau)\=q^{1/5}\left(1+\frac{1}{5}q+\frac{8}{25}q^2+\frac{41}{125}q^3
+\frac{284}{625}q^4-\frac{9301}{15625}q^5+\frac{32604}{78125}q^6+\cdots\right)\,,\\
&\f^{[15]}_6(\tau)\=q^{1/15}\left(1+\frac{1}{5}q+\frac{8}{25}q^2+\frac{41}{125}q^3
+\frac{284}{625}q^4+\frac{6324}{15625}q^5-\frac{29896}{78125}q^6+\cdots\right)\,,\\
&\f^{[15]}_7(\tau)\=1+\frac{1}{5}q+\frac{8}{25}q^2+\frac{41}{125}q^3+\frac{284}{625}q^4+\frac{6324}{15625}q^5
+\frac{48229}{78125}q^6-\cdots\,,
\end{aligned}
\end{equation}\label{eq.1/5}
The set~$\bigl\{\f^{[15]}_j\bigr\}$ spans a \textit{subspace} of~$M_{2/5}(\G(15), v_{15})$, and it is not difficult by using~\eqref{eq.q-series}
and~\eqref{eq.p=15} and the fact that the Strum bound of~$\G(15)$ is~1440 to find that
\begin{equation*}
\f^{[5]}_2(\tau)^2\=\f^{[15]}_4(\tau)-\f^{[15]}_7(\tau)\,,\quad
\f^{[5]}_2(\tau)\f^{[5]}_1(\tau)\= \f^{[15]}_3(\tau)\,,\quad
\f^{[5]}_1(\tau)^2\=\f^{[15]}_1(\tau)+\f^{[15]}_6(\tau)\,.
\end{equation*}
(cf.~p. 8, Lemma~3 in~\cite{I2}.), where the left-hand sides of the above equations are basis of  $M_{2/5}(\G(5), v_5^2)$. 
We now prove~$v_5^2=v_{15}$
since both hand sides of \eqref{eq.1/5} must have the same weights.
Since for any integer $p$ ($>3$) there are multiplier systems~$v_0$ corresponding $p=5, 15$ which we denote it by~$v_{0, 5}$ and~$v_{0, 15}$ by~$v_{0,5}(\gamma)=\frac{\eta^{3/5}(\gamma\thin\tau)}{(c\tau+d)^{3/10}\eta^{3/5}(\tau)}$ 
for any~$\gamma=\left(\begin{smallmatrix}a&b\\ c&d\end{smallmatrix}\right)\in\G(5)$ 
and~$v_{0,15}(\gamma)=\frac{\eta^{1/5}(\gamma\tau)}{(c\tau+d)^{1/10}\eta^{1/5}(\tau)}$
for any~$\gamma=\left(\begin{smallmatrix}a&b\\ c&d\end{smallmatrix}\right)\in\G(15)$, respectively. 
Then we immediately see~$v_{0,5}(\gamma)=v_{0,15}(\gamma)^3$ 
for any~$\gamma\in\G(15)$. It  also holds that $v_{0,15}^{40}=1$ for any~$\gamma\in\G(15)$ since~$\eta(\tau)^8=\left(\eta(\tau)^{1/5}\right)^{40}$ 
is a modular form of weight~4 
on~$\G(3)$. Thus for any~$\gamma\in\G(15)$ we can see that $v_p=v_0^{p^2-1}$ weight of $v_{15}$ is $2/5$.
\begin{equation*}
\frac{v_{15}(\gamma)}{v_{5}(\gamma)^2}\=\frac{v_{0,15}(\gamma)^{15^2-1}}{v_{0,5}(\gamma)^{2(5^2-1)}}
\=\frac{v_{0,15}(\gamma)^{15^2-1}}{v_{0,15}(\gamma)^{6(5^2-1)}}\=v_{0,15}(\gamma)^{80}\=\left(v_{0,15}(\gamma)^{40}\right)^{2}=1\,.
\end{equation*}
Therefore one can conclude~$M_{2m/5}(\G(5), v_{5}^{2m})\subset M_{2m/5}(\G(15),v_{15}^m)$ for any positive integer~$m$ which implies~(A).
The monic MLDE of the 7th order whose solution space consists of the Ibukiyama modular forms of weight~$2/5$  on~$\G(15)$; any function 
of~\eqref{eq.p=15} is a solution of the MLDE given by
\begin{multline*}
\sd^7_{2/5}(\f)-\frac{21}{50}E_4\sd^5_{2/5}(\f)+\frac{529}{2700}E_6\sd_{2/5}^4(\f)-\frac{3283}{90000}E_{8}\sd_{2/5}^3(\f)
+\frac{7733}{2430000}E_{10}\sd_{2/5}^2(\f)\\
-\frac{248501}{729000000}\Big(E_4^3-\frac{31233600}{22591}\Delta\Big)\sd_{2/5}(\f)+\frac{248501}{4374000000}E_{14}(\f)\=0\,.
\end{multline*}

Finally, we discuss modular forms of weight~$3/7$ on~$\G(21)$, i.e.~fractional modular forms associated to
the minimal model of type~$(2,21)$. Since the effective central charge is~$6/7$ the Ibukiyama modular forms are  given by 
\begin{alignat*}{5}
&\f_1^{[21]}(\tau)\=\eta(\tau)^{6/7}\ch_{2,21;1,1}(\tau)\,,&\quad&\f^{[21]}_2(\tau)\=\eta(\tau)^{6/7}\ch_{2,21;1,2}(\tau)\,,
&\quad&\f^{[21]}_3(\tau)\=\eta(\tau)^{6/7}\ch_{2,21;1,3}(\tau)\,,\\
&\f_4^{[21]}(\tau)\=\eta(\tau)^{6/7}\ch_{2,21;1,4}(\tau)\,, &\quad&\f^{[21]}_5(\tau)\=\eta(\tau)^{6/7}\ch_{2,21;1,5}(\tau)\,,
&\quad&\f^{[21]}_6(\tau)\=\eta(\tau)^{6/7}\ch_{2,21;1,6}(\tau)\,,\\
&\f_7^{[21]}(\tau)\=\eta(\tau)^{6/7}\ch_{2,21;1,7}(\tau)\,, &\quad&\f^{[21]}_8(\tau)\=\eta(\tau)^{6/7}\ch_{2,21;1,8}(\tau)\,,
&\quad&\f^{[21]}_9(\tau)\=\eta(\tau)^{6/7}\ch_{2,21;1,9}(\tau)\,,\\
&\f_{10}^{[21]}(\tau)\=\eta(\tau)^{6/7}\ch_{2,21;1,10}(\tau)\,.
\end{alignat*}
It is not difficult to see that~$\f^{[21]}_j(\tau)\=\eta(\tau)^{-3/21}\sum_{n\in\Z}(-1)^ne(\tau)^{\frac{21}{2}\left(n+\frac{2j-1}{42}\right)^2}$ 
for any~$1\leq j\leq10$, and then their Fourier expansions are given by
\begin{equation}\label{eq.p=10}
\begin{aligned}
&\f^{[21]}_1(\tau)\=q^{15/7}\bigl((1-\frac{6}{7}q+\frac{4}{49}q^2-\frac{1}{343}q^3+\frac{194}{2401}q^4-\frac{825}{16807}q^5+\frac{16351}{117649}q^6+\cdots\bigr)\,,\\
&\f^{[21]}_2(\tau)\=q^{12/7}\bigl(1+ \frac{1}{7}q-\frac{38}{49}q^2+\frac{27}{343}q^3+\frac{187}{2401}q^4+\frac{533}{16807}q^5+\frac{10576}{117649}q^6+\cdots\bigr)\,,\\
&\f^{[21]}_3(\tau)\=q^{4/3}\bigl(1+\frac{1}{7}q+\frac{11}{49}q^2-\frac{267}{343}q^3+\frac{383}{2401}q^4+\frac{484}{16807}q^5+\frac{20082}{117649}q^6+\cdots\bigr)\,,\\
&\f^{[21]}_4(\tau)\=q\bigl(1+\frac{1}{7}q+\frac{11}{49}q^2+\frac{76}{343}q^3-\frac{1675}{2401}q^4+\frac{1856}{16807}q^5+\frac{19739}{117649}q^6+\cdots\bigr)\,,\\
&\f^{[21]}_5(\tau)\=q^{5/7}\bigl(1+\frac{1}{7}q+\frac{11}{49}q^2+\frac{76}{343}q^3+\frac{726}{2401}q^4-\frac{12550}{16807}q^5+\frac{29343}{117649}q^6+\cdots\bigr)\,,\\
&\f^{[21]}_6(\tau)\=q^{10/21}\bigl(1+\frac{1}{7}q+\frac{11}{49}q^2+\frac{76}{343}q^3+\frac{726}{2401}q^4+\frac{4257}{16807}q^5-\frac{71499}{117649}q^6+\cdots\bigr)\,,\\
&\f^{[21]}_7(\tau)\=q^{2/7}\bigl(1+\frac{1}{7}q+\frac{11}{49}q^2+\frac{76}{343}q^3+\frac{726}{2401}q^4+\frac{4257}{16807}q^5+\frac{46150}{117649}q^6+\cdots\bigr)\,,\\
&\f^{[21]}_8(\tau)\=q^{1/7}\bigl(1+\frac{1}{7}q+\frac{11}{49}q^2+\frac{76}{343}q^3+\frac{726}{2401}q^4+\frac{4257}{16807}q^5+\frac{46150}{117649}q^6+\cdots\bigr)\,,\\
&\f^{[21]}_9(\tau)\=q^{1/21}\bigl(1+\frac{1}{7}q+\frac{11}{49}q^2+\frac{76}{343}q^3+\frac{726}{2401}q^4+\frac{4257}{16807}q^5+\frac{46150}{117649}q^6+\cdots\bigr)\,,\\
&\f^{[21]}_{10}(\tau)\=1+\frac{1}{7}q+\frac{11}{49}q^2+\frac{76}{343}q^3+\frac{726}{2401}q^4+\frac{4257}{16807}q^5+\frac{46150}{117649}q^6+\cdots\,.
\end{aligned}
\end{equation}
The monic MLDE such that each function of~\eqref{eq.p=10} is a solution is given by
\begin{multline*}
\sd^{10}_{3/7}(\f)-\frac{451}{336}E_4\sd^8_{3/7}(\f)+\frac{39545}{37044}E_6\sd_{3/7}^7(\f)-\frac{9757}{43904}E_{8}\sd_{3/7}^6(\f)
-\frac{4330919}{37340352}E_{10}\sd_{3/7}^5(\f)\\
+\frac{5130680923}{58549671936}\Big(E_4^3-\frac{3204764098560}{5130680923}\Delta\Big)\sd^4_{3/7}(\f)
-\frac{6729854203}{263473523712}E_{14} \sd^3_{3/7}(\f)\\
+\frac{1403994215923}{354108415868928}E_4\Big(E_4^3-\frac{1664226941276160}{1403994215923}\Delta\Big)\sd_{3/7}^2(\f)\\
-\frac{3323559970259}{13013484283183104}E_6\Big(E_4^3-\frac{1706576837437440}{3323559970259}\Delta\Big)\sd^{[1]}_{3/7}(\f)\\
-\frac{637460013}{42313823813632}E_4^2\Big(E_4^3-\frac{16388096}{12099}\Delta\Big)\f\=0\,.
\end{multline*}

 The set~$\bigl\{\f^{[21]}_j\bigr\}$ spans a \textit{subspace} of~$M_{3/7}(\G(21), v_{21})$, and it is not difficult to see that 
\begin{equation*}
\begin{split}
\f^{[7]}_3(\tau)^3&\=\f^{[21]}_{10}(\tau)^2+2\f^{[21]}_{10}(\tau)\f^{[21]}_{4}(\tau)-\f^{[21]}_{7}(\tau)\f^{[21]}_{5}(\tau)+\f^{[21]}_{4}(\tau)^2+\f^{[21]}_{7}(\tau)\f^{[21]}_{2}(\tau)\,,\\
\f^{[7]}_{3}(\tau)^2\f^{[7]}_{2}(\tau)&\=\f^{[21]}_{10}(\tau)\f^{[21]}_{8}(\tau)+\f^{[21]}_{10}(\tau)\f^{[21]}_{1}(\tau)+\f^{[21]}_{8}(\tau)\f^{[21]}_{4}(\tau)+\f^{[21]}_{4}(\tau)\f^{[21]}_{1}(\tau)\,,\\
\f^{[7]}_{3}(\tau)^2\f^{[7]}_{1}(\tau)&\=\f^{[21]}_{8}(\tau)\f^{[21]}_{7}(\tau)+\f^{[21]}_{7}(\tau)\f^{[21]}_{1}(\tau)\,,\\
\f^{[7]}_{2}(\tau)^3&\=\f^{[21]}_{8}(\tau)\f^{[21]}_{7}(\tau)+\f^{[21]}_{7}(\tau)\f^{[21]}_{1}(\tau)+\f^{[21]}_{5}(\tau)^2
-2\f^{[21]}_{5}(\tau)\f^{[21]}_{2}(\tau)+\f^{[21]}_{2}(\tau)^2\,,\\
\f^{[7]}_{2}(\tau)^2\f^{[7]}_{3}(\tau)&\=\f^{[21]}_{10}(\tau)^2+\f^{[21]}_{7}(\tau)\f^{[21]}_{4}(\tau)\,,\\
\f^{[7]}_{2}(\tau)^2\f^{[7]}_{1}(\tau)&\=\f^{[21]}_{10}(\tau)\f^{[21]}_{5}(\tau)-\f^{[21]}_{10}(\tau)\f^{[21]}_{2}(\tau)
+\f^{[21]}_{5}(\tau)\f^{[21]}_{4}(\tau)-\f^{[21]}_{4}(\tau)\f^{[21]}_{2}(\tau)\,,\\
\f^{[7]}_{1}(\tau)^3&\=\f^{[21]}_{10}(\tau)\f^{[21]}_{7}(\tau)-\f^{[21]}_{8}(\tau)^2-\f^{[21]}_{8}(\tau)\f^{[21]}_{1}(\tau)
+\f^{[21]}_{7}(\tau)\f^{[21]}_{4}(\tau)-\f^{[21]}_{1}(\tau)^2\,,\\
\f^{[7]}_{1}(\tau)^2\f^{[7]}_{3}(\tau)&\=\f^{[21]}_{8}(\tau)\f^{[21]}_{5}(\tau)-\f^{[21]}_{8}(\tau)\f^{[21]}_{2}(\tau)
+\f^{[21]}_{5}(\tau)\f^{[21]}_{1}(\tau)-\f^{[21]}_{2}(\tau)\f^{[21]}_{1}(\tau)\,\\
\f^{[7]}_{1}(\tau)^2\f^{[7]}_{2}(\tau)&\=\f^{[21]}_{7}(\tau)\f^{[21]}_{5}(\tau)-\f^{[21]}_{7}(\tau)\f^{[21]}_{2}(\tau)\,,\\
\f^{[7]}_{3}(\tau)\f^{[7]}_{2}(\tau)\f^{[7]}_{1}(\tau)&\=\f^{[21]}_{7}(\tau)^2
\end{split}
\end{equation*}
by using the Strum bound 
of~$\G(21)$.
We still need to prove that~$ v_{7}^{3}=v_{21}^2$. As before, we define two maps~$v_{0, 5}:\G(7)\to\C$ and~$v_{0, 15}:\G(21)\to\C$ 
by~$v_{0,7}(\gamma)=\frac{\eta^{3/7}(\gamma\tau)}{(c\tau+d)^{3/14}\eta^{3/7}(\tau)}$ 
for any~$\gamma=\left(\begin{smallmatrix}a&b\\c&d\end{smallmatrix}\right)\in\G(7)$ 
and~$v_{0,21}(\gamma)=\frac{\eta^{1/7}(\gamma\tau)}{(c\tau+d)^{1/14}\eta^{1/7}(\tau)}$ 
for any~$\gamma=\left(\begin{smallmatrix}a&b\\ c&d\end{smallmatrix}\right)\in\G(21)$, respectively.
Then it is  not difficult to see~$v_{0,7}(\gamma)^2=v_{0,21}(\gamma)^3$ for~$\gamma\in\G(21)$. We also obtain~$\left(v_{0,21}\right)^{56}=1$ 
for any~$\gamma\in\G(21)$ 
since~$\eta(\tau)^8=\left(\eta^{1/7}(\tau)\right)^{56}$ is a modular form of weight~4 on~$\G(3)$. Therefore  we have 
\begin{equation*}
\frac{v_{21}(\gamma)^2}{v_{7}(\gamma)^3}\=\frac{v_{0,21}(\gamma)^{2(21^2-1)}}{v_{0,7}(\gamma)^{3(7^2-1)}}
\=\frac{v_{0,21}(\gamma)^{2(21^2-1)}}{v_{0,21}(\gamma)^{9(7^2-1)}}\=v_{0,21}(\gamma)^{448}\=\left(v_{0,15}(\gamma)^{56}\right)^8=1\,.
\end{equation*}
Hence one can see~$M_{6m/7}(\G(7), v_{7}^{3m})\subseteq M_{6m/7}(\G(21),v_{21}^{2m})$ for any positive integer~$m$.

For odd integer~$p>13$ the set of the Ibukiyama modular forms is not always a basis of the space of modular forms 
of fractional weight~$(p-3)/2p$ on~$\G(p)$ with multiplier systems~$v_p$ (cf.~\cite{I2}), and we do not know 
if~(A) holds between~$M_*(\G(m), v_m)$ and~$M_*(\G(n), v_n)$ when $m, n>13$. 
However, it was already seen that the set of pairs $p=5, 15$ and $p=7, 21$ satisfy (A). This suggests;
 
\mn
\textbf{Conjecture}. Let~$\{\eta^{\ec}\ch_{\,2,\,p\,;\,1,\,s}\}$ and~$\{\eta^{\ec'}\ch_{\,2,\,p'\,;\,1,\,s'}\}$ be the sets of  modular forms 
on~$\G(p)$ and~$\G(p')$
with multiplier systems of the minimal models of type~$(2,p)$ and of type~$(2,p')$, respectively.  
Then the former and the latter have~the property~(A) if and only if~$p|p'$.

We now show that each set of the Ibykiyama modular forms spans the solution space of some monic modular linear differential equation.
We start with an easy lemma.

\begin{lemma}\label{L.Eta}
Let~$\G\subset PSL_2(\R)$ be a Fuchsian group. Suppose that~$f\in M_k(\Gamma)$ is a solution of the monic modular linear differential 
equation~$\sd_k^{n}(f)+\sum_{j=2}^nP_{2j}\sd_k^{n-j}(f)=0$. Then~$g=\eta^\ell f$ with an~$\ell\in\Z$ is a solution of a monic linear differential
equation~$\sd_{k+\frac{\ell}{2}}^{n}(g)+\sum_{j=2}^nP_{2j}\sd_{k+\frac{\ell}{2}}^{n-j}(g)=0$ of weight~$k+\frac{\ell}{2}$.
\end{lemma}
\begin{proof}
It follows from the definition of the iterated Serre operators and the equality~$\sd_k(\eta)=(1/24-k/12)E_2\eta$ 
that~$\sd_{k+\frac{\ell}{2}}^{n+1}(\eta^\ell f)=\sd_{k+\frac{\ell}{2}+2}^{n}\bigl(\sd_{k+\frac{\ell}{2}}(\eta^\ell f)\bigr)
=\sd_{k+\frac{\ell}{2}+2}^{n}\bigl(\eta^\ell\sd_k(f)\bigr)$since~$\sd_{k+\frac{\ell}{2}}^n(\eta^\ell  f)=\eta^\ell\sd_k^n(f)$.
\end{proof}

\begin{theorem}\label{t.rational}
 Let~$L_{2,\,p}$ be the minimal model of type~$(2,p\thin)$ with an odd integer~$p>3$, and let~$\ch_{\,2,\,p\,;\,1,\,s}$ for~$0<s\leq(p-1)/2$ 
 be the characters of the simple $L_{2,\,p}$-module~$L_{\,2,\,p\,;\,1,\,s}$. \\
\textup{(a)}
There exists a monic modular linear differential equation 
\begin{equation}\label{eq.mlde_frac}
\sd^{\frac{p-1}{2}}_0(f)+\sum_{j=2}^{(p-1)/2}P_{2j}\,\sd^{\frac{p-1}{2}-j}_0(f)\=0
\end{equation}
whose solution space is equal to the space spanned by the characters~$\ch_{\,2,\,p\,;\,1,\,(s+1)/2}$ for all~$0<s\leq(p-1)/2$, 
where~$P_{2j}$ is a holomorphic modular form of weight~$2j$ on the full modular group.\\
\textup{(b)}~The space of the Ibukiyama modular forms of weight~$\frac{p-3}{2p}$ on~$\Gamma(p)$  
is equal to the solution space of some monic modular linear differential equation of weight~$\frac{p-3}{2p}$ of the order~$\frac{p-1}{2}$ on the full modular group, which is written is the form
\begin{equation}\label{eq.mldech} 
\sd^{\frac{p-1}{2}}_{\frac{p-3}{2p}}(f)+\sum_{j=2}^{\frac{p-1}{2}}P_{2j}\,\sd^{\frac{p-3}{2p}-j}_{\frac{p-3}{2p}}(f)\=0\,,
\end{equation}
where each~$P_{2j}$ is the same modular form appearing in equation~\eqref{eq.mlde_frac}. 
\end{theorem}
\begin{proof}
The statement~(a) follows from and Theorem~\ref{t.exponents}. Lemma~\ref{L.Eta} with~$n=(p-1)/2$ and~$k=0$ shows
 the assertion~(b). 
 \end{proof} 
\noindent
It must be mentioned that equation~\eqref{eq.mlde_frac} was first found in the form which uses~$D=(2\pi i)^{-1}d/d\tau$ in Theorem 6.1 of~\cite{Milas}.  Since equation~\eqref{eq.mlde_frac} and Milas' linear differential equation are monic and have the same solution space, they must coincide 
and then the latter is automatically monic MLDEs.  We expected,  as an application of the Theorem~\ref{t.rational}~(b), 
that we can obtain supersingular $j$-invariant from the 0\thin th order term of equation~\eqref{eq.mldech} as in~\cite{MMO}. 
However, this fails since 0\thin th order terms of~{\eqref{eq.mldech}  and the Milas' one are equal.

\section{Vector-valued modular forms associated to modular forms of fractional weights on~$\mathbf{\G(5)}$}\label{s.monodromy}
Let~$f(\tau)$ be a (holomorphic or meromorphic) modular form of positive weight~$k$ on a Fuchsian group~$\G\subset PSL_2(\R)$ and~$t(\tau)$ a modular function with respect to~$\G$. Then the function~$\widetilde{f}(t)=f(t(\tau))$ satisfies a linear differential equation of order~$k+1$ 
with algebraic coefficients, or with polynomial coefficients if~$\G\backslash\H$ has genus~0 and~$t(\tau)$ generates the field 
of modular functions on~$\G$. 

Here we give examples of vector-valued modular forms of weight~$k=m/5$ with $m\in\Z_{\geq0}$ but of the minimal order of ordinary linear differential equation which has these modular forms as solutions with order~$>\text{weight}+1$. (See Proposition~21~p.61) of~\cite{123}).
The space of characters of the minimal model of type~$(2, 5)$ multiplied  by~$\eta^{2/5}$ and the space of modular forms of weight~$1/5$
with level five equal as we have seen in Theorem~\ref{t.minimal_Ibukiyama}.  We will take a symmetric tensor product representation of~$\MG$-actin $\pi_5 :M_{1/5}(\G(5),v_5)\to GL_2(\C)$, which give a VVMF, and the set spanned by components function of the VVMF  is equal to the solution space of some monic MLDE of order~$1+5k$.

Let~${}^t(g_1, g_2)\in\mathsf{Hol}(\H)\times\mathsf{Hol}(\H)$. We denote~$\pi_5(g_1, g_2)(\g):\MG\to GL_2(\C)$ a representation 
of~$\MG$ on~$GL_2(\C)$ defined by
for any~$\gamma\in\MG$. Then one can show easily that \blue(17)
\begin{equation}\label{eq.repd}
\pi_5{}^t\bigl(\f^{[5]}_1, \f^{[5]}_2\bigr)(T) 
\=\begin{pmatrix}
e\bigl(-\frac{2}{5}\bigr)&0\\ 
0&e{\bigl(-\frac{3}{5}\bigr)}
\end{pmatrix}\,,\;
\pi_5{}^t\bigl(\f^{[5]}_1, \f^{[5]}_2\bigr)(S)
\=\frac{e\left(\frac{3}{4}\right)}{\sqrt{10}}
\begin{pmatrix}
-\sqrt{5+\sqrt{5}}&\sqrt{5-\sqrt{5}}\\
\sqrt{5-\sqrt{5}}&\sqrt{5+\sqrt{5}}
\end{pmatrix}.
\end{equation}
}

It is proved in Lemma~1.7 in~\cite{I} that~$\dim  M_k(\Gamma(5))=1+5k$ for any~$k\in\frac{1}{5}\Z_{>0}$.
Let~$m+1=5k+1(=\dim  M_k(\G(5)))$, and define a VVMF~$\VF={}^t(f_0,\,\dotsc,\, f_m)$ with~$f_\ell=\bigl(\f^{[5]}_2\bigr)^{m-\ell}\bigl(\f^{[5]}_1\bigr)^{\ell}$ 
for any~$0\leq\ell\leq m$. Then it is shown in~\cite{NOA} that representation~\eqref{eq.repd} gives
\begin{equation*}
(\VF|_k\gamma)(\tau)\=A_\gamma\,\VF(\tau)\=
\begin{pmatrix}
a_0^{m+1}&(m+1)a_0^{m}b_0&\cdots&b_0^{m+1}\\
a_0^{m}c_0&\cdots&\cdots&b_0^{m}d_0\\
\cdots&\cdots&\cdots&\cdots\\
c_0^{m+1}&(m+1)c_0^{m}d_0&\cdots&d_0^{m+1}
\end{pmatrix}
\VF(\tau)\=\Sym^{m+1} {}^t (\pi_5(\f^{[5]}_2,\f^{[5]}_1)
(\gamma))\,\VF(\tau)
\end{equation*}
for any~$\gamma\in\MG$, 
where~$\pi_5{}^t\bigl(\f^{[5]}_1, \f^{[5]}_2\bigr)(\gamma)=\left(\begin{smallmatrix}a_0&b_0\\c_0&d_0\end{smallmatrix}\right)$.
More explicitly, the~$(m, n)$ entry of the matrix~$A_\gamma$ is given by
$\sum_{i+j=m+1-n}\binom{m+1-n}{j}\binom{n}{i}a_0^j\,b_0^i\,c_0^i\,d_0^{n-i}$.
Then the representation~$A_\gamma$ of~$\gamma\in\G(5)$ associated to~$\VF$ of~$\G(5)$ is the symmetric tensor 
representation~$\Sym^{m+1}(\pi_5(\f^{[5]}_1, \f^{[5]}_2))$ of degree~$1+m$. 

We now study the monic MLDE  which is associated to the vector-valued modular form~$\VF$. A short calculation shows
$\bigl(\f^{[5]}_2\bigr)^{5k-\ell}\bigl(f^{[5]}_1\bigr)^{\ell}=e(\tau)^{\ell/5}+O\left(e(\tau)^{(\ell+1)/5}\right)$ for any~$0\leq\ell\leq5k$. 
Since~$\dim M_k (\Gamma(5),v_5^{1+5k})=1+5k$ for any~$k\in\frac{1}{5}\Z_{\geq0}$ and t is proved in Lemma~1.7 in~\cite{I} that~$\dim  M_k(\Gamma(5))=1+5k$ for any~$k\in\frac{1}{5}\Z_{>0}$.
{(31)~the leading exponents 
of~$\bigl(\f^{[5]}_2\bigr)^{5k-\ell}\bigl(\f^{[5]}_1\bigr)^\ell$ are~$\ell/5$}, the sum of $\ell/5$  multiplied by~12 is equal to 
\begin{equation*}
12\sum_{\ell=0}^{5k}\dfrac{\ell}{5}=6k(5k+1)\=d(d+k-1)\,,\quad d\=1+5k\,.
\end{equation*}
Therefore there exists some monic MLDE whose solution space is equal to the space~$M_k(\Gamma(5), v_5^{1+5k})$.  Hence   we have

\begin{theorem} 
Let~$m$ is~$k=\frac{1}{5}\Z_{>}0$. There exists a monic modular linear differential equation of weight~$k$ and order~$1+5k$ 
on $\MG$, whose solution space coincides with the space of modular forms of weight~$k$ with multipliers $v_5^{5k}$ on~$\G(5)$ for any~$k$. 
\end{theorem}

For the case of the characters of the minimal model of type~$(2,5)$,  the symmetric  representation and the MLDE are known (cf.~\cite{Milas}).
\section{Appendix}
We now complete the proof of Theorem~\ref{t.n}. The case $\ell=0$ was done already, and then we show the theorem for $0<\ell<12$.

\begin{table}[h]\label{table:np2}
\begin{center}
\begin{tabular}{c||c|c|c|c}
$\ell$&$1$&$2$&$3$&$4$\\ \hline\hline
$\frac{12(2r+3)}{\gcd(4,r)\gcd(3,r)}$&$12(24m+5)$&$6(24m+7)$&$4(24m+9)$&$12(24m+11)$\medskip\\ 
$\ell$&$5$&$6$&$7$&$8$\\ \hline\hline
$\frac{12(2r+3)}{\gcd(4,r)\gcd(3,r)}$
&$12(24m+13)$&$4(24m+15)$&$2(24m+17)$&$12(24m+19)$\medskip\\ 
$\ell$&$9$&$10$&$11$&$12\,(\ell=0)$\\ \hline\hline
$\frac{12(2r+3)}{\gcd(4,r)\gcd(3,r)}$&$4(24m+21$&$6(24m+23$&$12(24m+25)$&$12p$\\
\end{tabular}
\caption{The values of Theorem~\ref{t.n}}
\end{center}
\end{table}



\begin{thebibliography}{}

\bibitem{ANS} Y.~Arike, K.~Nagatomo and Y.~Sakai, Vertex operator supper algebras and the theta group, preprint.
\bibitem{AM} G.~Anderson, G. Moore, Rationality in conformal field theory, Commmu.~Math.~Phys., \textbf{117}, 441-450 (1988)
\bibitem{NOA}
N.~O.~Azaiez, Formes quasi-modulaires sur des groupes modulaires co-compacts et restrictions des formes modulaires 
de Hilbert aux courbes modulaires (2005) \url{https://tel.archives-ouvertes.fr/tel-00011122/document}

\bibitem{DLN}
C.~Dong,  X.~Lin and S.~H.~Ng,
Congruence property in conformal field theory, Algebra and Number Theory, \textbf{9}, no. 9, 2121--2166 (2015)

\bibitem{FK}
H. Farkas and I. Kra, \textit{Riemann surfaces}, First edition, Springer Verlag, New York, 1981.

\bibitem{FeF2}
B.~Feigin and D.~Fuchs,
\textit{Verma modules over the Virasoro algebra}, Lect.~Notes in Math., \textbf{1060}, 230--245 (1982)

\bibitem{FeF1}
B.~Feigin and D.~Fuchs, 
\textit{Representation theory of the Virasoro algebra,} In Representation theory of Lie groups and related topics,
Adv.~Stud.~Contemp.~Math., \textbf{7}, New York, Gordon and Breach, 465--554, 1990.

\bibitem{FZ}
I.~Frenkel and Y.~Zhu,
Vertex operator algebras associated to representations of affine and Virasoro algebras,
Duke Math.~J., \textbf{66}, No.~1, 123--168 (1992)

\bibitem{I2}
T.~Ibukiyama, Graded rings of modular forms of rational weight,  Research in Number Theory 6.1, 1--13, (2020)


\bibitem{I}
T.~Ibukiyama, 
Modular forms of rational weights and modular varieties,  Abhandlungen aus dem Mathematischen Seminar 
der Universit\"at Hamburg, \textbf{70}, no.~1,  315--339 (2000)

\bibitem{IK} 
K.~Iohara and Y.~Koga, 
\textit{Representation theory of the Virasoro algebra}, 
Springer-Verlag London, 2011.

\bibitem{Ince}
E.~L.~Ince,
\textit{Ordinary differential equations}, Dover Publications, New York, 1944.

\bibitem{K}
V.~G.~Kac, A.~K.~Raina and N.~Rozhkovskaya,
\textit{Bombay Lectures on highest weight representations of infinite dimensional Lie algebras}, Second Edition, 
World Scientific Publishing Co. Pte. Ltd. 2013.
\bibitem{LL}
J.~Lepowsky and H.-S.~Li,
\textit{Introduction to vertex operator algebras and their representations}, 
Prog.~Math., Birkh{\"a}user, Boston, 2004.

\bibitem{Mason}
G.~Mason,
Vector-valued modular forms and linear differential operators, 
Int.~J.~Number~Theory, \textbf{3}, no.3, 377--390 (2007)


\bibitem{MN}
A.~Matsuo and K.~Nagatomo,  
\textit{Axioms for a vertex algebra and the locality of quantum fields}, 
MSJ Memoirs, \textbf{4}, Mathematical Society of Japan, 1999.

\bibitem{Milas1}
A.~Milas,
Ramanujan's ``Lost Notebook" and the Virasoro algebra, 
Commun.~Math.~Phys., \textbf{251}, 567--588 (2004)

\bibitem{Milas}
 A.~Milas, 
Virasoro algebra, Dedekind $\eta$-function and specialized MacDonald identities, 
Transformation Groups, \textbf{9}, No. 3, 273--288 (2004)

\bibitem{MMO}
A.~Milas, E.~Mortenson and K.~Ono,
Number theoretic properties of Wronskians of Andrews-Gordon series, 
Int.~J.~Number Theory, \textbf{4}, no.~2, 323--337 (2008)


\bibitem {123} D.~Zagier, Elliptic modular forms and their applications. In
J.~Bruinier, G.~Harder, G.~van der Geer and \hbox{D.~Zagier}, {\it The 1--2--3 of Modular Forms\,: 
Lectures at a Summer School in Nordfjordeid, Norway} (ed.~K.~Ranestad).
Universitext, Springer-Verlag, Berlin--Heidelberg--New York (2008), 1--103. 


\bibitem{Zhu}
Y.~Zhu,
Modular invariance of characters of vertex operator algebras,
J.~Amer.~Math.~Soc., \textbf{9}, 237--302 (1996)

\end{thebibliography}
\end{document}